\documentclass[10pt]{amsart}
\usepackage{graphicx}
\usepackage{amscd}
\usepackage{amsmath}
\usepackage{amsthm}
\usepackage{amsfonts}
\usepackage{amssymb}
\usepackage{mathrsfs}
\usepackage{bm}
\usepackage{enumerate}
\usepackage{amsrefs}
\usepackage{xcolor}
\usepackage[colorlinks, citecolor=blue, linkcolor=red, pdfstartview=FitB]{hyperref}
\usepackage[latin1]{inputenc}

\usepackage{tikz}
\usepackage{tikz-cd}
\usepackage{caption}
\usetikzlibrary{decorations.markings}
\usepackage{lipsum}
\usepackage{mathrsfs}

\usepackage{marginnote}	
\usepackage{soul} 			

\newcommand{\bB}{{\mathbb{B}}}
\newcommand{\bC}{{\mathbb{C}}}

\newcommand{\bM}{{\mathbb{M}}}

\newcommand{\bS}{{\mathbb{S}}}

  \newcommand{\B}{{\mathcal{B}}}

\renewcommand{\S}{{\mathcal{S}}}

\newcommand{\fK}{{\mathfrak{K}}}

\newcommand{\rC}{\mathrm{C}}

\renewcommand{\phi}{\varphi}

\newcommand{\upchi}{{\raise.35ex\hbox{$\chi$}}}

\newcommand{\ol}{\overline}

\newcommand{\mycomment}[1]{}

\newcommand{\qand}{\quad\text{and}\quad}

\newcommand{\id}{\operatorname{id}}

\newcommand{\re}{\operatorname{Re}}

\newcommand{\at}{\operatorname{at}}

\newtheorem{lemma}{Lemma}[section]
\newtheorem{theorem}[lemma]{Theorem}
\newtheorem{proposition}[lemma]{Proposition}
\newtheorem{corollary}[lemma]{Corollary}
\newtheorem{theoremx}{Theorem}

\theoremstyle{definition}

\newtheorem{example}{Example}

\author{Rapha\"el Clou\^atre}
\address{Department of Mathematics, University of Manitoba, Winnipeg, Manitoba, Canada R3T 2N2}
\email{raphael.clouatre@umanitoba.ca\vspace{-2ex}}

\thanks{R.C. was partially supported by an NSERC Discovery Grant.}

\title[A new obstruction to Arveson's conjecture]{A new obstruction to Arveson's hyperrigidity conjecture}
\begin{document}
\begin{abstract}
Let $A$ be a unital $\rC^*$-algebra containing a closed two-sided ideal $J$ and an operator system $X$. We enlarge $X$ to an operator system $\S(X,J)$  in $\bM_2(A)$, and show that in order for $\S(X,J)$ to be hyperrigid,  each $*$-represen-\ tation of $\rC^*(X)$ annihilating $\rC^*(X)\cap J$ must admit a unique contractive completely positive extension from $X$ to the larger $\rC^*$-algebra $\rC^*(X)+J$.  We leverage this implicit additional rigidity constraint to construct counterexamples to Arveson's hyperrigidity conjecture. A key condition in our construction is the mutual orthogonality of the atomic projection of $\rC^*(X)$ and the support projection of $J$, which we  interpret as a new obstruction to the conjecture. Specializing to the case where $J$ is the ideal of compact operators on a Hilbert space, we recover as a by-product of our general construction the recent counterexample of Bilich and Dor-On. On the other hand, we find that such a pathology cannot be implemented using our construction when $A$ admits only finite-dimensional irreducible $*$-representations, thereby illustrating that the obstruction only manifests itself in  noncommutative settings.
\end{abstract}
\maketitle

\section{Introduction}

Let $A$ be a unital $\rC^*$-algebra and let $X\subset A$ be an operator system. Denote by $\rC^*(X)$ the $\rC^*$-algebra that $X$ generates inside $A$. Recall that a unital $*$-representation $\pi:\rC^*(X)\to B(H)$ is said to have the \emph{unique extension property} with respect to $X$ if, given a unital completely positive map $\psi:\rC^*(X)\to B(H)$ agreeing with $\pi$ on $X$, it must necessarily follow that $\pi$ and $\psi$ agree everywhere on $\rC^*(X)$. If, in addition, $\pi$ is irreducible, then it is said to be a \emph{boundary representation} for $X$. Originally introduced by Arveson in \cite{arveson1969}, these representations have become a mainstay in modern approaches to studying operator systems and nonselfadjoint operator algebras. Indeed, they can be used to compute $\rC^*$-envelopes \cite{dritschel2005boundary},\cite{arveson2008noncommutative},\cite{davidson2015choquet}, and they can be meaningfully  interpreted as the Choquet boundary in noncommutative convexity \cite{kleski2014boundary},\cite{davidson2019noncommutative}.

The motivation for this paper is a conjecture formulated by Arveson in 2011 \cite{arveson2011noncommutative}, in analogy with some classical results in approximation theory \cite{korovkin1953convergence},\cite{vsavskin1967mil}.  We say that $X$ is \emph{hyperrigid} in $\rC^*(X)$ if all unital $*$-representations of $\rC^*(X)$ have the unique extension property with respect to $X$. The so-called \emph{hyperrigidity conjecture} is that in order for $X$ to be hyperrigid, it suffices that all irreducible $*$-representations of $\rC^*(X)$ be boundary representations for $X$.  Much work has been done on this problem since its formulation, see for instance \cite{kleski2014korovkin},\cite{kennedy2015essential},\cite{clouatre2018non},\cite{clouatre2018unperforated},\cite{davidson2021choquet},\cite{clouatre2023boundary},\cite{PS2024I},\cite{PS2024II},\cite{PS2024III}  and the references therein.

Significant breakthroughs related to this conjecture occurred in last two years, and the fog surrounding the problem has begun to clear \cite{CTh2024},\cite{scherer2024},\cite{scherer2025},\cite{bilich2025}. Of particular relevance for our current endeavour is \cite{BDO2024}, where the first counterexample to the conjecture was found. Therein, a singly generated unital operator algebra generating an essentially abelian $\rC^*$-algebra is explicitly constructed and shown to violate the conjecture. That such an otherwise well-behaved object could disprove the conjecture motivates the search for a hidden obstruction. The aim of this paper is to uncover such an obstruction, by means of the following general construction inspired by  \cite{BDO2024}.

Let $J\subset A$ be a closed two-sided ideal. To avoid degeneracies we assume that $J$ is not contained in $\rC^*(X)$. We 
define the operator system $\S(X,J)\subset \bM_2(A)$ to consist of elements of the form
\begin{equation}\label{Eq:SXJ}
\begin{bmatrix}
x & j \\ k & b
\end{bmatrix} \quad \text{where } x\in X,j\in J,k\in J \text{ and }b\in J+\bC 1.
\end{equation}
Let $\B(X,J)$ denote the $\rC^*$-algebra that $\S(X,J)$ generates in $\bM_2(A)$.
In Section \ref{S:uep}, we examine the unique extension property with respect to $\S(X,J)$ for $*$-representations  of $\B(X,J)$. We show in Lemma \ref{L:JinB} that $\B(X,J)$ consists of elements of the form $\begin{bmatrix} a & b \\ c & d\end{bmatrix}$ with $a\in \rC^*(X)+J$, $b\in J,c\in J$ and $d\in J+\bC 1$. In particular, $\B(X,J)$ contains the closed two-sided ideal $\bM_2(J)$. 
Then, by basic representation theory of $\rC^*$-algebras, a  $*$-representation of $\B(X,J)$ splits as a direct sum of two representations: one annihilating $\bM_2(J)$, and one that is nondegenerate   on $\bM_2(J)$. We show in Theorem \ref{T:uepnondeg} that all $*$-representations that are nondegenerate on $\bM_2(J)$ have the unique extension property with respect to $\S(X,J)$. Thus, from the point of view of hyperrigidity and Arveson's conjecture, there only remains to analyze those $*$-representations annihilating $\bM_2(J)$. 

Assume  then that we are given a unital $*$-representation $\Pi:\B(X,J)\to B(H)$ annihilating $\bM_2(J)$. A crucial observation that we make in Lemma \ref{L:singrep} is that there must exist a unique $*$-representation  $\pi:\rC^*(X)+J\to B(H)$ for which the pair $(\Pi,\pi)$ is $(X,J)$-\emph{adapted} in the sense that
\[
\Pi\left(\begin{bmatrix}
a& b \\ c & d
\end{bmatrix}\right)-\chi(d)1=\pi(a-\chi(d)1)
\]
for each $\begin{bmatrix}
a& b \\ c & d
\end{bmatrix}\in \B(X,J)$. (Here, $\chi:J+\bC 1\to \bC$ denotes the quotient map corresponding to $J$.) In Theorem \ref{T:singrepuep}, we completely characterize when $\Pi$ has the unique extension property with respect to $\S(X,J)$  in terms of $\pi$. Unexpectedly, the requirement is that  $\pi|_{X}$ must admit a unique contractive completely positive extension not only to $\rC^*(X)$, but to the larger $\rC^*$-algebra $\rC^*(X)+J$. This ``hidden" additional rigidity requirement is the key to producing counterexamples to Arveson's conjecture. As a first indication of this pathology, we show in Corollary \ref{C:intnothyperrigid} that the operator system $\S(X,J)$ is not hyperrigid in $\B(X,J)$ whenever $\rC^*(X)\cap J=\{0\}$.

In order for $\S(X,J)$ to serve as counterexample to the conjecture, however, we must also verify that all irreducible $*$-representations of $\B(X,J)$ are boundary representations for $\S(X,J)$. This is more delicate, and in Section \ref{S:obst} we undertake the task of completely characterizing when this occurs. For this purpose, we consider the atomic projection $z^X_{\at}$ of $\rC^*(X)$ (viewed as an element of $\rC^*(X)^{\perp\perp}\subset A^{**}$) and examine its relationship with the support projection $p_J\in A^{**}$ of $J$. Theorem \ref{T:projcriterion} gives the desired characterization in terms of these projections. As a consequence of Theorem \ref{T:projcriterion}, we show in Corollary \ref{C:liminal} that when all irreducible $*$-representations of $A$ are finite-dimensional, then there is at least one irreducible $*$-representation of $\B(X,J)$ without the unique extension property.  In particular, this means that a counterexample to the conjecture can only ever be implemented, at least using the construction at hand, in situations that are somewhat far from being commutative. Thus, this machinery does not shed new light on the commutative version of Arveson's conjecture, which is still open.   A further consequence of our characterization is Corollary \ref{C:counterex}, which is the main result of the paper and reads as follows.

\begin{theoremx}\label{T:A}
Let $A$ be a unital $\rC^*$-algebra containing a closed two-sided ideal $J$ and an operator system $X$.  Assume that 
\begin{enumerate}[{\rm (i)}]
\item $\rC^*(X)$ and $J$ have trivial intersection,
\item the projections $z^X_{\at}$ and $p_J$ are mutually orthogonal in $A^{**}$, and
\item all the irreducible $*$-representations of $\rC^*(X)$ are boundary representations for $X$.
\end{enumerate}
Then, $\S(X,J)\subset \B(X,J)$ is a counterexample to Arveson's hyperrigidity conjecture.
\end{theoremx}

In particular, we interpret the mutual orthogonality of $z^X_{\at}$ and $p_J$ to be an obstruction to Arveson's conjecture.

Finally, in Section \ref{S:compact}, we specialize our results to the case where $J$ is the ideal of compact operators on some Hilbert space. In this case, we exhibit a structural condition on the von Neumann algebra $\rC^*(X)''$ that is sufficient for the obstruction to dissipate (Theorem \ref{T:counterexK}). As an application, we explain in Example \ref{E:BDO} how to recover the counterexample found in \cite{BDO2024} from our machinery. We also obtain the following streamlined mechanism for producing counterexamples to Arveson's conjecture (Corollary \ref{C:ampcounterex}). Here, we use the notation
\[
X_\infty=\left\{\bigoplus_{n=1}^\infty x:x\in X\right\}\subset B(H^{(\infty)}).
\]

\begin{theoremx}\label{T:B}
Let $X\subset B(H)$ be an operator system. Assume that 
\begin{enumerate}[{\rm (i)}]
\item the von Neumann algebra $\rC^*(X)''$ contains no type $I$ factor as a direct summand, and
\item all the irreducible $*$-representations of $\rC^*(X)$ are boundary representations for $X$.
\end{enumerate}
Then, $\S(X_\infty,\fK(H^{(\infty)}))\subset \B(X_\infty,\fK(H^{(\infty)})$ is a counterexample to Arveson's hyperrigidity conjecture.
\end{theoremx}

\section{The unique extension property for representations of $\B(X,J)$}\label{S:uep}

Let $A$ be a unital $\rC^*$-algebra containing a closed two-sided ideal $J$ and an operator system $X$. Let $\S(X,J)\subset \bM_2(A)$ denote the operator system constructed in \eqref{Eq:SXJ}. We let $\B(X,J)\subset \bM_2(A)$ denote the $\rC^*$-algebra generated by $\S(X,J)$.  The goal of this section is to fully describe the $*$-representations of $\B(X,J)$ with the unique extension property with respect to $\S(X,J)$. We start with a preliminary fact describing the elements of $\B(X,J)$.

\begin{lemma}\label{L:JinB}
The $\rC^*$-algebra $\B(X,J)$ consists of elements of the form $\begin{bmatrix} a & b \\ c & d\end{bmatrix}$ with $a\in \rC^*(X)+J$, $b\in J, c\in J$ and $d\in J+\bC 1$.
\end{lemma}
\begin{proof}
It is readily checked that the set of elements of the form $\begin{bmatrix} a & b \\ c & d\end{bmatrix}$ with $a\in \rC^*(X)+J$, $b\in J,c\in J$ and $d\in J+\bC 1$, is a $\rC^*$-algebra containing $\S(X,J)$, and thus it contains $\B(X,J)$.

Conversely, let $\begin{bmatrix} a & b \\ c & d\end{bmatrix}$ with $a\in \rC^*(X)+J$, $b\in J,c\in J$ and $d\in J+\bC 1$. Then,
$
\begin{bmatrix} 0 & b\\ c & d\end{bmatrix}\in \S(X,J)
$
by construction, so it only remains to show that $
\begin{bmatrix} a & 0\\ 0 & 0\end{bmatrix}\in \B(X,J).
$
This clearly holds if $a\in \rC^*(X)$, again by construction of $\S(X,J)$.  Thus, we  assume that $a\in J$. Let $(e_\lambda)$ be a contractive approximate unit for $J$, so that the net $(ae_\lambda)$ converges in norm to $a$. For each $\lambda$, we see that
\begin{align*}
\begin{bmatrix} ae_\lambda & 0 \\ 0 & 0\end{bmatrix}=\begin{bmatrix} 0 & a \\ 0 & 0\end{bmatrix}\begin{bmatrix} 0 & 0 \\ e_\lambda & 0\end{bmatrix}\in \B(X,J)
\end{align*}
since  both $\begin{bmatrix} 0 & a \\ 0 & 0\end{bmatrix}$ and $\begin{bmatrix} 0 & 0 \\ e_\lambda & 0\end{bmatrix}$ lie in $\S(X,J)$, so upon taking the limit we get the desired conclusion.
\end{proof}

One consequence of the previous result is that $\B(X,J)$ contains the closed two-sided ideal $\bM_2(J)$. By basic representation theory of $\rC^*$-algebras, any $*$-representation of $\B(X,J)$ is a direct sum of two representations, one that is nondegenerate on $\bM_2(J)$, and another that annihilates $\bM_2(J)$. Since the unique extension property is preserved by direct sums \cite[Proposition 4.4]{arveson2011noncommutative}, we may analyze these two types of representations separately. 

\subsection{The nondegenerate representations}
To handle the representations that are nondegenerate on $\bM_2(J)$, we require two standard preliminary  facts. Recall  that given a $\rC^*$-algebra $A$, by means of the universal representation we may view the bidual $A^{**}$ as a von Neumann algebra. 

\begin{lemma}\label{L:Stinespring}
Let $A$ be a unital $\rC^*$-algebra and let $J\subset A$ be a closed two-sided ideal with a positive, increasing, contractive approximate unit $(e_\lambda)$. Let $\psi:A\to B(H)$ be a unital completely positive map with the property that the net $(\psi(e_\lambda))$ converges to $1$ in the strong operator topology. Then, the following statements hold.
\begin{enumerate}[{\rm (i)}]
\item  Let $\sigma:A\to B(K)$ be a unital $*$-representation and $V:H\to K$ be an isometry such that 
\[
\psi(a)=V^*\sigma(a)V, \quad a\in A.
\] 
Then,  $\lim_\lambda \|\sigma(e_\lambda) V\xi-V\xi\|=0$ for every $\xi\in H.$ In particular, $VH\subset \ol{\sigma(J)VH}$.
\item Let  $\phi:A\to B(H)$ be a unital completely positive map agreeing with $\psi$ on $J$. Then, $\phi=\psi$.
\end{enumerate}
\end{lemma} 
\begin{proof}
(i) Observe that 
\[
\psi(e_\lambda)^2\leq \psi(e_\lambda^2)\leq 1
\]
by the Schwarz inequality. By assumption, the net $(\psi(e_\lambda))$ converges to $1$ in the strong operator topology, so that $(\psi(e_\lambda)^2)$ converges to $1$ in the weak operator topology. The same is then true of the net $(\psi(e_\lambda^2))$, by the previous inequalities. For $\xi\in H$ we find
\begin{align*}
\|\sigma(e_\lambda)V\xi-V\xi\|^2&=\langle V^*\sigma(e_\lambda^2)V\xi,\xi \rangle-2\re \langle V^* \sigma(e_\lambda)V\xi,\xi\rangle+\|V\xi\|^2\\
&=\langle \psi(e_\lambda^2)\xi,\xi\rangle -2\re \langle \psi(e_\lambda)\xi,\xi\rangle +\|\xi\|^2
\end{align*}
so indeed
\[
\lim_\lambda \|\sigma(e_\lambda) V\xi-V\xi\|=0.
\]
This in turn guarantees that $VH\subset \ol{\sigma(J)VH}$.

(ii) Let $\widehat\phi:A^{**}\to B(H)$ and $\widehat\psi:A^{**}\to B(H)$ denote the unique weak-$*$ continuous extensions of $\phi$ and $\psi$, respectively. These maps are still completely positive. If we let $q\in A^{**}$ denote the weak-$*$ limit of the increasing contractive net $(e_\lambda)$, then we see that $\widehat \phi(q)=\widehat \psi(q)=1$ since $\phi$ agrees with $\psi$ on $J$.
In particular, $\widehat\phi$ and $\widehat\psi$ contain $q$ in their respective multiplicative domains. Let $a\in A$.  Then, $aq$ lies in the weak-$*$ closure of $J$ inside of $A^{**}$, so $\widehat \phi(aq)=\widehat\psi(aq)$  and
\[
\phi(a)=\widehat\phi(aq)=\widehat \psi(aq)=\psi(a).
\]
We conclude that $\phi=\psi$.
\end{proof}

\begin{lemma}\label{L:repM2}
Let $A$ be a $\rC^*$-algebra and let $\pi:\bM_2(A)\to B(H)$ be a non-degenerate $*$-representation. Then, there is a non-degenerate $*$-representation $\pi_0:A\to B(H_0)$ such that $\pi$ is unitarily equivalent to the ampliation $\pi_0^{(2)}$. 
\end{lemma}
\begin{proof}
Let  $\widehat\pi:\bM_2(A)^{**}\to B(H)$ denote the unique weak-$*$ continuous extension of $\pi$, which is still a $*$-representation. Throughout the proof we make the implicit identification $ \bM_2(A)^{**}\cong \bM_2(A^{**})$.
Let $(e_\lambda)$ be an  positive, increasing, contractive approximate unit for $A$, and let $q\in A^{**}$ denote its weak-$*$ limit, which is necessarily  a projection.  Note also that $(e_\lambda\oplus e_\lambda)$ is a contractive approximate identity for $\bM_2(A)$.  Since $\pi$ is assumed to be non-degenerate, we have that $H=\ol{\pi(\bM_2(A))H}$ and hence the net $(\pi(e_\lambda\oplus e_\lambda))$ converges to $1$ in the strong operator topology.  Hence, we must have $\widehat\pi(q\oplus q)=1$ so that
\[
u=\widehat\pi\left(\begin{bmatrix} 0 & q \\ q & 0 \end{bmatrix}\right)\in B(H)
\]
is a self-adjoint unitary. 
Consider the projection
\[
p=\widehat\pi\left(\begin{bmatrix} q & 0 \\ 0 & 0 \end{bmatrix}\right)\in B(H)
\]
and let $H_0\subset H$ denote its range. It is readily verified that $ p u =u(1-p)$ so we may define a unitary operator $W:H=H_0\oplus H_0^\perp\to H_0\oplus H_0$ as
\[
W=\begin{bmatrix} 1 & 0 \\ 0 & u|_{H_0}\end{bmatrix}.
\]
Define $\pi_0:A\to B(H_0)$ as
\[
\pi_0(a)=p\pi\left(\begin{bmatrix} a & 0 \\ 0 & 0 \end{bmatrix}\right)|_{H_0}, \quad a\in A.
\]
It is clear that $\widehat\pi_0(q)=1$ so $\pi_0$ is non-degenerate. Finally, a standard verification reveals that 
\[
W\pi\left(\begin{bmatrix} a & b \\ c & d \end{bmatrix}\right)W^*=\begin{bmatrix} \pi_0(a) & \pi_0(b) \\ \pi_0(c) & \pi_0(d) \end{bmatrix}
\]
for each $a,b,c,d\in A$.
\end{proof}

We can now analyze the unique extension property for representations of $\B(X,J)$ that are nondegenerate on $\bM_2(J)$.

\begin{theorem}\label{T:uepnondeg}
Let $A$ be a unital $\rC^*$-algebra containing a closed two-sided ideal $J$ and an operator system $X$.
Let $\pi:\B(X,J)\to B(H)$ be a unital  $*$-representation such that $\pi|_{\bM_2(J)}$ is nondegenerate. Then, $\pi$ has the unique extension property with respect to $\S(X,J)$.
\end{theorem}
\begin{proof}
Let $\psi:\B(X,J)\to B(H)$ be a unital completely positive map agreeing with $\pi$ on $\S(X,J)$. Let $(\sigma,V,K)$ be a Stinespring dilation of $\psi$, so that $V:H\to K$ is an isometry, $\sigma:\B(X,J)\to B(K)$ is a unital $*$-representation and we have
\[
\psi(b)=V^*\sigma(b)V, \quad b\in \B(X,J)
\]
Let $(e_\lambda)$ be a positive, increasing, contractive approximate unit for $J$. Then, $(e_\lambda\oplus e_\lambda)$ is  a positive, increasing, contractive approximate unit for $\bM_2(J)$. For each $\lambda$, we have $e_\lambda^{1/2}\in J$ so the element
\[
x_\lambda=\begin{bmatrix} 0 & e_\lambda^{1/2}\\ 0 & 0 \end{bmatrix}
\]
lies in $\S(X,J)$.
Using that $\psi$ agrees with $\pi$ on $\S(X,J)$, by means of the Schwarz inequality we obtain
\begin{align*}
1&\geq \psi(e_\lambda\oplus e_\lambda)=\psi(x_\lambda x_\lambda^*+x_\lambda^*x_\lambda)\\
&\geq \psi(x_\lambda)\psi(x_\lambda)^*+\psi(x_\lambda)^*\psi(x_\lambda)= \pi(x_\lambda)\pi(x_\lambda)^*+\pi(x_\lambda)^*\pi(x_\lambda)\\
&=\pi(x_\lambda x_\lambda^*+x_\lambda^*x_\lambda)=\pi(e_\lambda\oplus e_\lambda).
\end{align*}
Since $\pi$ is assumed to be non-degenerate on $\bM_2(J)$ (i.e. $\ol{\pi(\bM_2(J))H}=H$), it follows that the net $(\pi(e_\lambda\oplus e_\lambda))$ converges to $1$ in the strong operator topology, and thus the same is true for the net $ (\psi(e_\lambda\oplus e_\lambda))$ in view of the preceding inequalities. Applying Lemma \ref{L:Stinespring}(i) we infer that  $VH\subset \ol{\sigma(\bM_2(J))VH}$.  Hence, upon compressing $\sigma$ to the closed reducing subspace $\ol{\sigma(\bM_2(J))VH}$, we may assume  that the net $(\sigma(e_\lambda\oplus e_\lambda))$ converges to $1$ in the strong operator topology and that $\sigma|_{\bM_2(J)}$ is nondegenerate.

By Lemma \ref{L:repM2}, there are nondegenerate $*$-representations $\pi_0:J\to B(H_0)$ and $\sigma_0:J\to B(K_0)$ such that, up to unitary equivalence, $\pi|_{\bM_2(J)}=\pi_0^{(2)}$ and $\sigma|_{\bM_2(J)}=\sigma_0^{(2)}$. According to the decompositions $H=H_0\oplus H_0$ and $K=K_0\oplus K_0$, we may write 
\[
V=\begin{bmatrix}
V_{11} & V_{12}\\ V_{21} & V_{22}
\end{bmatrix}.
\]
For $a\in J$, we know that $\begin{bmatrix} 0 & a \\ 0 & 0 \end{bmatrix}\in \S(X,J)$, so that
\begin{align*}
\begin{bmatrix}
0 & \pi_0(a)\\ 0 & 0 
\end{bmatrix}&=\pi\left(\begin{bmatrix} 0 & a \\ 0 & 0 \end{bmatrix} \right)=\psi\left(\begin{bmatrix} 0 & a \\ 0 & 0 \end{bmatrix} \right)\\
&=V^* \sigma\left(\begin{bmatrix} 0 & a \\ 0 & 0 \end{bmatrix}\right)V=V^* \begin{bmatrix} 0 & \sigma_0(a)\\ 0 & 0 \end{bmatrix}V\\
&=\begin{bmatrix}
V_{11}^*\sigma_0(a)V_{21} & V_{11}^*\sigma_0(a)V_{22}\\ V_{12}^*\sigma_0(a)V_{21} & V_{12}^*\sigma_0(a)V_{22}
\end{bmatrix}
\end{align*}
and in particular we find 
\begin{equation}\label{Eq:V}
\pi_0(a)=V_{11}^*\sigma_0(a)V_{22}.
\end{equation}
The nets $(\pi_0(e_\lambda ))$ and $(\sigma_0(e_\lambda))$ converge to the respective identity operators on $H_0$ and $K_0$,  in the strong operator topology. Using \eqref{Eq:V}, we find
\[
\pi_0(e_\lambda)=V_{11}^*\sigma_0(e_\lambda)V_{22}
\]
for each $\lambda$, so upon taking the limit we find $1=V_{11}^* V_{22}$. 

Next, we compute using that $V_{11}^*V_{11}\leq 1$ and $V_{22}^*V_{22}\leq 1$,
\begin{align*}
(V_{11}-V_{22})^*(V_{11}-V_{22})&=V_{11}^* V_{11}-V_{11}^* V_{22}-V_{22}^* V_{11}+V_{22}^* V_{22}\\
&\leq 2-V_{11}^* V_{22}-V_{22}^* V_{11}=0
\end{align*}
hence $V_{11}=V_{22}$ are isometries. Since $V$ itself is an isometry, this readily implies that $V_{12}=V_{21}=0$. For $a\in J$, we now find by \eqref{Eq:V} again
\begin{align*}
\psi\left(\begin{bmatrix} a & 0 \\ 0 & 0 \end{bmatrix} \right)&=V^* \sigma\left(\begin{bmatrix} a  & 0 \\ 0 & 0 \end{bmatrix}\right)V=V^* \begin{bmatrix} \sigma_0(a)& 0\\ 0 & 0 \end{bmatrix}V\\
&=\begin{bmatrix}
V_{11}^*\sigma_0(a)V_{22} & 0\\ 0 & 0
\end{bmatrix}=\begin{bmatrix}
\pi_0(a)& 0\\ 0 & 0
\end{bmatrix}\\
&=\pi\left(\begin{bmatrix} a & 0 \\ 0 & 0 \end{bmatrix} \right).
\end{align*}
Since $\pi$ and $\psi$ already agree on $\S(X,J)$, we conclude that $\pi$ and $\psi$ agree on $\bM_2(J)$. By Lemma \ref{L:Stinespring}(ii), we can further infer that $\pi$ and $\psi$ agree everywhere on $\B(X,J)$.
\end{proof}

We remark here that when $J$ is chosen to be the ideal of compact operators on a Hilbert space, then the previous result follows immediately from the construction of $\S(X,J)$ and Arveson's boundary theorem \cite{arveson1972}. Thus, in this instance our general construction requires a much finer analysis than that found in \cite{BDO2024}.

\subsection{The annihilating representations}

Next, we turn to the unique extension property for representations of $\B(X,J)$ annihilating $\bM_2(J)$.  Our approach is predicated on a connection between such $*$-representations of $\B(X,J)$ and $*$-representations of $\rC^*(X)+J$.  To streamline the exposition, we introduce some notation and terminology. We let $\chi:J+\bC 1\to \bC$ denote the quotient map by $J$. Let $\Pi:\B(X,J)\to B(H)$ and $\pi:\rC^*(X)+J\to B(H)$ be  $*$-representations. 
Given $\begin{bmatrix}
a& b \\ c & d
\end{bmatrix}\in \B(X,J)$, recall that $a\in \rC^*(X)+J$, $b\in J,c\in J$ and $d\in J+\bC 1$ by Lemma \ref{L:JinB}.
We then say that the pair $(\Pi,\pi)$ is \emph{$(X,J)$-adapted} if 
\begin{equation}\label{Eq:adapted}
\Pi\left(\begin{bmatrix}
a& b \\ c & d
\end{bmatrix}\right)-\chi(d)1=\pi(a-\chi(d)1)
\end{equation}
for each $\begin{bmatrix}
a& b \\ c & d
\end{bmatrix}\in \B(X,J)$. In particular, in that case we see that the commutants of $\Pi(\B(X,J))$ and $\pi(\rC^*(X)+J)$ in $B(H)$ are the same:
\begin{equation}\label{Eq:commutant}
\Pi(\B(X,J))'=\pi(\rC^*(X)+J)'.
\end{equation}

\begin{lemma}\label{L:singrep}
Let $A$ be a unital $\rC^*$-algebra containing a closed two-sided ideal $J$ and an operator system $X$. The following statements hold.
\begin{enumerate}[{\rm (i)}]
\item Let $\Pi:\B(X,J)\to B(H)$ be a unital $*$-representation annihilating $\bM_2(J)$.
Then, there is a unique $*$-representation $\pi:\rC^*(X)+J\to B(H)$ annihilating $J$ such that the pair $(\Pi,\pi)$ is $(X,J)$-adapted.

\item Let $\pi:\rC^*(X)+J\to B(H)$ be a unital $*$-representation annihilating $J$. Then, there is a unique unital $*$-representation $\Pi:\B(X,J)\to B(H)$ annihilating $\bM_2(J)$ such that the pair $(\Pi,\pi)$ is $(X,J)$-adapted.
\end{enumerate}

\end{lemma}
\begin{proof}
Uniqueness clearly follows from \eqref{Eq:adapted} for both statements, so we focus on existence below.

(i)  Let $\begin{bmatrix}
a& b \\ c & d
\end{bmatrix}\in \B(X,J)$. Then, $d-\chi(d)1 \in J$ and
$\begin{bmatrix}
0& b \\ c & d-\chi(d)1
\end{bmatrix}\in  \bM_2(J)$. 
Using that $\Pi$ is unital and annihilates $\bM_2(J)$, we find
\begin{equation}\label{Eq:pipi0}
\Pi\left(\begin{bmatrix}
a& b \\ c & d
\end{bmatrix}\right)
=\Pi\left(\begin{bmatrix}
a& 0 \\ 0 & \chi(d)1
\end{bmatrix}\right)=\Pi\left(\begin{bmatrix}
a-\chi(d)1& 0 \\ 0 & 0
\end{bmatrix}\right)+\chi(d)1.
\end{equation}
Observe also that the map $\tau:\rC^*(X)+J\to \B(X,J)$ defined as
\[
\tau(a)=\begin{bmatrix}
a& 0 \\ 0 & 0
\end{bmatrix}, \quad a\in \rC^*(X)+J
\]
is an injective $*$-representation.
Define $\pi:\rC^*(X)+J\to B(H)$ as $\pi=\Pi\circ \tau$.  By construction, it is clear that $\pi$ annihilates $J$ since $\Pi$ annihilates $\bM_2(J)$. The fact that the pair $(\Pi,\pi)$ is $(X,J)$-adapted follows immediately from \eqref{Eq:pipi0}.

(ii) Let $Q:\B(X,J)\to \B(X,J)/\bM_2(J)$  and $q:\rC^*(X)+J\to (\rC^*(X)+J)/J$ denote the quotient maps. 
It is readily verified that the map 
\[
\begin{bmatrix} a&  b\\ c & d\end{bmatrix}+\bM_2(J)\mapsto a+J
\]
is a well-defined unital $*$-homomorphism on  $\B(X,J)/\bM_2(J)$ whose range is $q(\rC^*(X)+J)$, by Lemma \ref{L:JinB}. Let us denote this map by $\sigma:\B(X,J)/\bM_2(J)\to (\rC^*(X)+J)/J$. 
By assumption, there is a unital $*$-representation $\widetilde\pi:(\rC^*(X)+J)/J\to B(H)$ such that $\widetilde\pi\circ q=\pi$. Define $\Pi=\widetilde\pi\circ \sigma\circ Q:\B(X,J)\to B(H)$. Clearly, $\Pi$ annihilates $\bM_2(J)$, and for $\begin{bmatrix} a & b \\ c& d\end{bmatrix}\in \B(X,J)$ we find
\begin{align*}
\Pi\left(\begin{bmatrix} a & b \\ c& d\end{bmatrix} \right)&=(\widetilde\pi\circ q)(a)=\pi(a).
\end{align*}
Since both $\Pi$ and $\pi$ are unital, this implies \eqref{Eq:adapted}, so indeed the pair $(\Pi,\pi)$ is $(X,J)$-adapted.
\end{proof}

In the next development, we will need to work with maps on a unital $\rC^*$-algebra that are not necessarily unital. We record a simple fact for ease of reference below.

\begin{lemma}\label{L:nonunital}
Let $A$ be a unital $\rC^*$-algebra and let $\phi:A\to B(H)$ be a contractive completely positive map. Let $e=\phi(1)$. Then,
\[
\phi(a)=\phi(a)e=e\phi(a), \quad a\in A.
\]
\end{lemma}
\begin{proof}
By \cite[Theorem 4.1]{paulsen2002completely}, there is a unital $*$-representation $\sigma:A\to B(K)$ and a contraction $V:H\to K$ with the property that
\[
\phi(a)=V^*\pi(a)V, \quad a\in A. 
\]
For $a\in A$ we thus find
\begin{align*}
\phi(a)^*\phi(a)&=V^*\pi(a)VV^*\pi(a)V\leq  V^*\pi(a^*a)V\\
&= \phi(a^*a)\leq \|a\|^2 \phi(1)=\|a\|^2 e
\end{align*}
and likewise $\phi(a)\phi(a)^*\leq \|a\|^2 e$. This implies
 that $\phi(a)(1-e)=(1-e)\phi(a)=0$.
\end{proof}

We can now analyze the unique extension property for $*$-representations of $\B(X,J)$ that annihilate $\bM_2(J)$, thereby complementing Theorem \ref{T:uepnondeg}.

\begin{theorem}\label{T:singrepuep}
Let $A$ be a unital $\rC^*$-algebra containing a closed two-sided ideal $J$ and an operator system $X$. Let $\Pi:\B(X,H)\to B(H)$ be a unital $*$-representation annihilating $\bM_2(J)$, and let $\pi:\rC^*(X)+J\to B(H)$ be the unique $*$-representation such that the pair $(\Pi,\pi)$ is $(X,J)$-adapted. Then, the following statements are equivalent.
\begin{enumerate}[{\rm (i)}]
\item $\Pi$ has the unique extension property with respect to $\S(X,J)$.
\item $\pi$ is the  unique contractive completely positive extension of $\pi|_X$ to $\rC^*(X)+J$.
\end{enumerate}
\end{theorem}
\begin{proof}
(ii)$\Rightarrow$(i): Let $\Phi:\B(X,J)\to B(H)$ be a unital completely positive map agreeing with $\Pi$ on $\S(X,J)$.
By Lemma \ref{L:JinB}, we may define a contractive completely positive map $\phi:\rC^*(X)+J\to B(H)$ as
\[
\phi(a)=\Phi\left(\begin{bmatrix}
a& 0 \\ 0 & 0
\end{bmatrix}\right), \quad a\in \rC^*(X)+J.
\]
Since the pair $(\Pi,\pi)$ is $(X,J)$-adapted,  \eqref{Eq:adapted} implies that $\phi$ agrees with $\pi$ on $X$, so our assumption forces $\phi$ and $\pi$ to agree on $\rC^*(X)+J$.

 Fix $\begin{bmatrix}
a & b \\ c& d
\end{bmatrix}\in \B(X,J)$. By Lemma \ref{L:JinB}, we infer that $a\in \rC^*(X)+J$, $b\in J,c\in J$ and $d\in J+\bC 1$. Thus, the previous paragraph implies that
\begin{equation}\label{Eq:piphia}
\pi(a-\chi(d)1)=\phi(a-\chi(d)1).
\end{equation}
Moreover the elements
$\begin{bmatrix}
0& 0 \\ 0 & d-\chi(d)1
\end{bmatrix},\begin{bmatrix}
0 & 0 \\ 0 & b^*b
\end{bmatrix},\begin{bmatrix}
0 & 0 \\ 0 & cc^*
\end{bmatrix}$ all lie in $ \bM_2(J)\cap \S(X,J)$
so
\[
\Phi\left(\begin{bmatrix}
0 & 0 \\  0& d-\chi(d)1
\end{bmatrix}\right)=\Phi\left(\begin{bmatrix}
0 & 0 \\ 0 & b^*b
\end{bmatrix} \right)=\Phi\left(\begin{bmatrix}
0 & 0 \\ 0 & cc^*
\end{bmatrix} \right)=0
\]
since  $\Phi$ and $\Pi$ agree on $\S(X,J)$ and $\Pi$  annihilates $\bM_2(J)$.
By the Schwarz inequality, this implies that
\[
\Phi\left(\begin{bmatrix}
0& b \\ 0& 0
\end{bmatrix}\right)=\Phi\left(\begin{bmatrix}
0& 0 \\ c& 0
\end{bmatrix}\right)=0
\]
whence
\begin{equation}\label{Eq:phi}
\Phi\left(\begin{bmatrix}
a& b \\ c & d
\end{bmatrix}\right)=\Phi\left(\begin{bmatrix}
a& 0 \\ 0 & \chi(d)1
\end{bmatrix}\right)=\Phi\left(\begin{bmatrix}
a-\chi(d)1& 0 \\ 0 & 0
\end{bmatrix}\right)+\chi(d)1.
\end{equation}
Employing Equations \eqref{Eq:adapted},\eqref{Eq:piphia} and \eqref{Eq:phi} we find
\begin{align*}
\Pi\left(\begin{bmatrix}
a& b \\ c & d
\end{bmatrix}\right)-\chi(d) 1&=\pi(a-\chi(d)1)=\phi(a-\chi(d)1)\\
&=\Phi\left(\begin{bmatrix}
a-\chi(d)1& 0 \\ 0 & 0
\end{bmatrix}\right)\\
&=\Phi\left(\begin{bmatrix}
a& b \\ c & d
\end{bmatrix}\right)-\chi(d)1.
\end{align*}
We conclude that $\Pi$ and $\Phi$ agree on $\B(X,J)$, so indeed $\Pi$ has the unique extension property with respect to $\S(X,J)$.

(i)$\Rightarrow$(ii):   Let $\phi:\rC^*(X)+J \to B(H)$ be a contractive completely positive map agreeing with $\pi$ on $X$. By Lemma \ref{L:nonunital}, the projection $p=\pi(1)$ commutes with the image of $\phi$.
Invoking Lemma \ref{L:JinB}, we can then define a \emph{unital} completely positive map $\Phi:\B(X,J)\to B(H)$ as
\begin{equation}\label{Eq:phi0}
\Phi\left(\begin{bmatrix}
a& b \\ c & d
\end{bmatrix}\right)=\phi(a)p+\chi(d)(1-p)
\end{equation}
 for $\begin{bmatrix}
a& b \\ c & d
\end{bmatrix}\in \B(X,J)$.
Using \eqref{Eq:adapted} we may also write
\begin{align*}
\Pi\left( \begin{bmatrix}
a& b \\ c & d
\end{bmatrix}\right)
&=\pi(a)p+\chi(d)(1-p).
\end{align*}
It then follows from \eqref{Eq:phi0} that $\Phi$ and $\Pi$ agree on $\S(X,J)$, so by assumption $\Phi$ and $\Pi$ agree on $\B(X,J)$. 
Finally, for $a\in \rC^*(X)+J$ we have $\begin{bmatrix}
a& 0 \\ 0 & 0
\end{bmatrix}\in  \B(X,J)$ by  Lemma \ref{L:JinB}, so an application of Lemma \ref{L:nonunital} and Equation \eqref{Eq:phi0} implies
\begin{align*}
\phi(a)&=\phi(a)p=\Phi\left(\begin{bmatrix}
a& 0 \\ 0 & 0
\end{bmatrix}\right)
=\Pi\left(\begin{bmatrix}
a& 0 \\ 0 & 0
\end{bmatrix}\right)=\pi(a).
\end{align*}
We conclude that indeed $\pi$ is the unique contractive completely positive extension of $\pi|_X$ to $\rC^*(X)+J$.
\end{proof}

We reiterate here the key point in the previous result: the unique extension property for $\Pi$ with respect to $\S(X,J)$ is encoded in the fact that $\pi|_{X}$ must admit a unique extension \emph{all the way to} $\rC^*(X)+J$, and not simply to $\rC^*(X)$ as one would expect. This is the phenomenon that can be exploited in constructing nonhyperrigid operator systems, as we illustrate next.

\begin{corollary}\label{C:intnothyperrigid}
Let $A$ be a unital $\rC^*$-algebra containing a closed two-sided ideal $J$ and an operator system $X$. Assume $\rC^*(X)\cap J=\{0\}$. Then, there is a unital  $*$-representation of $\B(X,J)$ annihilating $\bM_2(J)$ that does not have the unique extension property with respect to $\S(X,J)$. 
\end{corollary}
\begin{proof}
As above, let $q:\rC^*(X)+J\to (\rC^*(X)+J)/J$ denote the quotient map.
Because $\rC^*(X)\cap J=\{0\}$,  there is a $*$-isomorphism $\theta:(\rC^*(X)+J)/J\to \rC^*(X)$ satisfying $\theta\circ q=\id$ on $\rC^*(X)$. Put $\pi=\theta\circ q$, which is a unital $*$-homomorphism on $\rC^*(X)+J$ annihilating $J$. We may assume that $\rC^*(X)\subset B(H)$ for some Hilbert space $H$. By Lemma \ref{L:singrep}, there is a unique unital $*$-representation $\Pi:\B(X,J)\to B(H)$ annihilating $\bM_2(J)$ such that $(\Pi,\pi)$ is $(X,J)$-adapted.
Since $\pi=\id$ on $\rC^*(X)$, the identity map on $\rC^*(X)+J$ is a contractive completely positive extension of $\pi|_X$, which is clearly distinct from $\pi$ as it does not annihilate $J$. Hence, $\Pi$ does not have the unique extension property with respect to $\S(X,J)$ by virtue of Theorem \ref{T:singrepuep}.
\end{proof}

\section{An obstruction to Arveson's hyperrigidity conjecture}\label{S:obst}

In this section, we leverage the information gathered above in order to produce counterexamples to Arveson's conjecture. In the setting of Corollary \ref{C:intnothyperrigid}, what remains to be done is to guarantee that all irreducible $*$-representations be boundary representations. We will identify sufficient conditions for this to occur, thereby highlighting a certain obstruction to the conjecture.

\subsection{Bidual preliminaries}
 Let $A$ be a $\rC^*$-algebra.  Given a bounded linear map $\phi:A\to B(H)$, we let $\widehat\phi:A^{**}\to B(H)$ denote its unique weak-$*$ continuous extension. When, $\phi$ is multiplicative, then so is $\widehat\phi$. 

 Let $\Omega$ be a set of unitary equivalence classes of irreducible $*$-representations of $A$. For each $\omega\in \Omega$, choose an irreducible $*$-representation $\sigma_\omega:A\to B(K_\omega)$ in the class $\omega$. Let $p_\omega\in A^{**}$ denote the support projection of $\widehat{\sigma_\omega}$, so that $\ker \widehat{\sigma_\omega}=A^{**}(1-p_\omega)$. Since each $\sigma_\omega$ is irreducible, there is no non-zero proper central projection in $A^{**}$ dominated by $p_\omega$. Hence, given $\omega,\omega'\in \Omega$ such that $p_\omega p_{\omega'}\neq 0$, we must have $p_\omega=p_{\omega'}$. In turn, because the representations are irreducible, this implies that $\omega=\omega'$ by \cite[Corollary 3.8.10]{pedersen2018}. Hence, the set $\{p_\omega:\omega\in \Omega\}$ consists of pairwise mutually orthogonal projections and we define
\begin{equation}\label{Eq:zOmega}
z_\Omega=\bigoplus_{\omega\in \Omega}p_\omega
\end{equation}
which is a central projection in $A^{**}$. We now record a useful property of this projection.

\begin{lemma}\label{L:atomic}
Let $A$ be a $\rC^*$-algebra and let $\Omega$ be a non-empty subset of unitary equivalence classes of irreducible $*$-representations. Let $\pi:A\to B(H)$ be a non-degenerate $*$-representation. Then, the following statements are equivalent.
\begin{enumerate}[{\rm (i)}]
\item There is a non-empty subset $\Omega_\pi\subset \Omega$ and a set of cardinal numbers $\{\kappa_\omega:\omega\in \Omega_\pi\}$ such that $\pi$ is unitary equivalent to $\bigoplus_{\omega\in \Omega_\pi}\sigma^{(\kappa_\omega)}_\omega$.
\item $\widehat\pi(z_\Omega)=1$.
\end{enumerate}
\end{lemma}
\begin{proof}
(i)$\Rightarrow$(ii): This is trivial.

(ii)$\Rightarrow$(i): Let $\Omega_\pi$ be the set consisting of those $\omega\in \Omega$ for which $\widehat\pi(p_\omega)\neq 0$. Note that $\Omega_\pi$ is non-empty since 
\begin{equation}\label{Eq:pizomega}
1=\widehat\pi(z_\Omega)=\bigoplus_{\omega\in \Omega}\widehat\pi(p_\omega).
\end{equation}
For $\omega\in \Omega_\pi$, we let $q_\omega=\widehat\pi(p_\omega)$, which is a non-zero projection in $B(H)$ whose range $H_\omega$ is reducing for $\pi$. By virtue of \eqref{Eq:pizomega}, we may write $\pi=\bigoplus_{\omega\in \Omega_\pi} \pi_\omega$, where $\pi_\omega:A\to B(H_\omega)$ is the compression of $\pi$ to $H_\omega$. For $\omega\in \Omega_\pi$, we have $\widehat{\pi_\omega}(p_\omega)=1$, so that $p_\omega$ dominates the support projection of $\pi_\omega$. On the other hand, we saw before the lemma that $p_\omega$ is a minimal central projection, so it must coincide with the support projection of $\pi_\omega$. Therefore, the map $\theta_\omega: \sigma_\omega(A)''\to \pi_\omega(A)''$ defined as
\[
\theta_\omega(\widehat{\sigma_\omega}(x))=\widehat{\pi_\omega}(x),\quad x\in A^{**}
\]
is  a weak-$*$ continuous $*$-isomorphism. Since $\sigma_\omega$ is irreducible, we infer that $\theta_\omega$ is a unital, weak-$*$ continuous $*$-representation of $B(K_\omega)$ into $B(H_\omega)$.  We conclude that there is a cardinal number $\kappa_\omega$ and a unitary $U_\omega: H_\omega\to K_\omega^{(\kappa_\omega)}$ such that
\[
\theta_\omega(\widehat{\sigma_\omega}(x))= U_\omega^*\widehat{\sigma_\omega}(x)^{(\kappa_\omega)} U_\omega
\]
for every $x\in A^{**}$. 
Hence $\pi_\omega$ is unitarily equivalent to some ampliation of $\sigma_\omega$, as desired.
\end{proof}

There are two particular choices of sets $\Omega$ that will be relevant for us. First, when $\Omega$ is the entire spectrum of $A$ (i.e. the set of all unitary equivalence classes of irreducible $*$-representations), then the resulting projection $z^A_{\at}\in A^{**}$ is the so-called \emph{atomic} projection of $A$. 

Assume now that $B\subset A$ is a $\rC^*$-subalgebra. Then, the atomic projection $z_{\at}^B$ may be viewed as a projection in $B^{\perp\perp}\subset A^{**}$, using the standard identification $B^{**}\cong B^{\perp\perp}$. This projection can be compared with the atomic projection of $A$ by means of the next result.

\begin{lemma}\label{L:atcomp}
Let $A$ be a $\rC^*$-algebra and $B\subset A$ be a $\rC^*$-subalgebra. Then, the following statements are equivalent. 
\begin{enumerate}[{\rm (i)}]
\item If $\pi:A\to B(H)$ is an irreducible $*$-representation, then $\pi|_{B}$ is a direct sum of irreducible $*$-representations on $B$.
\item $z_{\at}^B\geq z^A_{\at}$.
\end{enumerate}
\end{lemma}
\begin{proof}
(i)$\Rightarrow$(ii): Choose a $*$-representation   $\pi:A\to B(H)$  such that $\widehat \pi(z_{\at}^A)=1$. By Lemma \ref{L:atomic}, we see that $\pi$ is a direct sum of irreducible $*$-representations. By assumption, this implies that $\pi|_B$ is also a direct sum of irreducible $*$-representations of $B$, so $\widehat\pi(z^B_{\at})=1$ by another application of Lemma \ref{L:atomic}. We conclude that $z_{\at}^B\geq z^A_{\at}$.

(ii)$\Rightarrow$(i):  Let $\pi:A\to B(H)$ be an irreducible $*$-representation. Then, $\widehat\pi(z_{\at}^A)=1$ by definition of the atomic projection. By assumption, we must have $\widehat\pi(z_{\at}^B)=1$ as well, so that $\pi|_B$ is a direct sum of irreducible $*$-representations of $B$ by Lemma \ref{L:atomic}.
\end{proof}

Let us now describe the second case where the construction \eqref{Eq:zOmega} is of particular relevance to us. Let $J\subset A$ be a closed two-sided ideal. Let $\Omega_J$ denote the set of unitary equivalence classes of irreducible $*$-representations of $A$ annihilating $J$. 

We wish to establish a connection between the projection $z_{\Omega_J}$ defined by \eqref{Eq:zOmega} and the \emph{support} projection of $J$, that is the central projection $p_J\in A^{**}$ satisfying $J^{\perp\perp}=A^{**}p$.

\begin{lemma}\label{L:projequal}
We have $z_{\Omega_J}=z^A_{\at}(1-p_J)$.
\end{lemma}
\begin{proof}
Plainly we have  $z_{\Omega_J}\leq z^A_{\at}$. Moreover, for every $[\sigma]\in \Omega_J$ we have $J^{\perp\perp}\subset \ker \widehat\sigma$ whence $p_J\leq 1-z_{\Omega_J}$. Since $z^A_{\at}$ and $p_J$ are central, we conclude that $z_{\Omega_J}\leq z_{\at}(1-p_J)$. 

In proving the reverse inequality, we may assume that the projection $z_{\at}(1-p_J)$ is non-zero. Hence, there is a $*$-representation $\pi:A\to B(H)$ such that $\widehat\pi(z_{\at}(1-p_J))=1$. We must show that $\widehat{\pi}(z_{\Omega_J})=1$.

On one hand, we see that $\widehat\pi(z_{\at})=1$, so by Lemma \ref{L:atomic} there is a set $\{\sigma_\lambda:\lambda\in \Lambda\}$ of irreducible $*$-representations of $A$ such that $\pi=\bigoplus_{\lambda\in \Lambda}\sigma_\lambda$. On the other hand, we also have $\widehat\pi(p_J)=0$ so $\widehat{\sigma_\lambda}(p_J)=0$ or $J^{\perp\perp}\subset \ker \widehat{\sigma_\lambda}$ for every $\lambda\in \Lambda$. We conclude that $[\sigma_\lambda]\in \Omega_J$ so $\widehat{\sigma_\lambda}(z_{\Omega_J})=1$ for every $\lambda\in \Lambda$, which in turn implies that $\widehat{\pi}(z_{\Omega_J})=1$.
\end{proof}

\subsection{When are all irreducibles boundary representations?}
We now return to our usual setting. Let $A$ be a unital $\rC^*$-algebra containing a closed two-sided ideal $J$ and an operator system $X$.  We record another elementary preliminary fact.

\begin{lemma}\label{L:charOmegaJ}
Let $\pi:\rC^*(X)\to B(H)$ be an irreducible $*$-representation. Then,  $\pi$ annihilates $J\cap \rC^*(X)$  if and only if $\pi$ admits an extension to an irreducible $*$-representation of $\rC^*(X)+J$ that annihilates $J$.
\end{lemma}
\begin{proof}
Assume that $\pi$ annihilates  $J\cap \rC^*(X)$. Let \[Q:\rC^*(X)+J\to (\rC^*(X)+J)/J \qand q:\rC^*(X)\to \rC^*(X)/(\rC^*(X)\cap J)\]
 denote the natural quotient maps, and let $$\theta:(\rC^*(X)+J)/J\to  \rC^*(X)/(\rC^*(X)\cap J)$$ denote the usual $*$-isomorphism such that $\theta\circ Q|_{\rC^*(X)}=q$. Now, there is an irreducible $*$-representation $\widetilde \pi:q(\rC^*(X))\to B(H)$ such that $\widetilde \pi\circ q=\pi$. It is readily verified that the $*$-representation $\widetilde \pi \circ \theta\circ Q:\rC^*(X)+J\to B(H)$ is irreducible, extends $\pi$, and annihilates $J$.  The other direction is trivial.
\end{proof}

In what follows, we  will write $z^X_{\at}$ for the atomic projection of $\rC^*(X)$ and $p^X_J$ for the support projection of the closed two-sided ideal $J\cap \rC^*(X)$. We view  both of these as projections in $(\rC^*(X))^{\perp\perp}\subset A^{**}$. We also let  $p_J\in A^{**}$ denote the support projection of the ideal $J$. Since $(J\cap \rC^*(X))^{\perp\perp}\subset J^{\perp\perp}$, we always have that 
\begin{equation}\label{Eq:pJX}
p^X_J\leq p_J.
\end{equation}
The next result is a key technical tool.

\begin{proposition}\label{P:projzero}
Let $A$ be a unital $\rC^*$-algebra containing a closed two-sided ideal $J$ and an operator system $X$. 
Then, the following statements are equivalent.
\begin{enumerate}[{\rm (i)}]
\item The projections $z^X_{\at}(1-p^X_J)$ and $p_J$ are mutually orthogonal in $A^{**}$.
\item If $\theta:A\to B(H)$ is a pure unital completely positive map such that $\theta|_{\rC^*(X)}$ is an irreducible $*$-representation annihilating $\rC^*(X)\cap J$, then $\theta$ annihilates $J$.
\item If $\psi:A\to B(H)$ is a unital completely positive map such that $\psi|_{\rC^*(X)}$ is an irreducible $*$-representation annihilating $\rC^*(X)\cap J$, then $\psi$ annihilates $J$.
\end{enumerate}
\end{proposition}
\begin{proof}
Throughout the proof, we let $\Omega_J$ denote the set of unitary equivalence classes of irreducible $*$-representations of $\rC^*(X)$ annihilating the closed two-sided ideal $\rC^*(X)\cap J$. Let $z^X_{\Omega_J}\in ( \rC^*(X))^{**}$ be the projection given by \eqref{Eq:zOmega}. Then, $z^X_{\at}(1-p^X_J)=z^X_{\Omega_J}$, by virtue of Lemma \ref{L:projequal}.

(i)$\Rightarrow$(ii): Let $\pi:\rC^*(X)\to B(H)$ be an irreducible $*$-representation annihilating $\rC^*(X)\cap J$. Then,  $[\pi]\in \Omega_J$. In particular, $\widehat\pi(z^X_{\Omega_J})=1$.  Let $\theta:A\to B(H)$ be a pure unital completely positive map extending $\pi$.

Assume towards a contradiction that $\widehat\theta(p_J)\neq 0$. An application of \cite[Corollary 1.4.3]{arveson1969} shows that the minimal Stinespring representation $\sigma$ of $\theta$ must be irreducible. Since $p_J$ is central and $\widehat\sigma(p_J)\neq 0$, we find $\widehat\sigma(p_J)=1$ so $\widehat\theta(p_J)=1$ as well.
 On the other hand, $\widehat \theta(z^X_{\Omega_J})=\widehat \pi(z^X_{\Omega_J})=1$ since $z^X_{\Omega_J}\in (\rC^*(X))^{\perp\perp}\subset A^{**}$ and $\theta$ agrees with $\pi$ on $\rC^*(X)$. Hence, both $p_J$ and $z^X_{\Omega_J}$ lie in the multiplicative domain of $\widehat\theta$, so $\widehat\theta(z^X_{\Omega_J}p_J)=1$ which is impossible since $z^X_{\Omega_J}p_J=0$ by assumption.

(ii)$\Rightarrow$(iii):  Let $\pi:\rC^*(X)\to B(H)$ be an irreducible $*$-representation annihilating $\rC^*(X)\cap J$.  Let $E$ denote the set of unital completely positive maps $\psi:A \to B(H)$ extending $\pi$. Clearly, $E$ is convex and compact in the pointwise weak-$*$ topology. By the Krein--Milman theorem, it suffices to fix an extreme point $\theta$ of $E$ and to show that $\theta$ annihilates $J$. 

For convenience, we denote by $P(\rC^*(X))$ the convex set of unital completely positive maps $\rC^*(X)\to B(H)$. We define $P(A)$ in a similar fashion.  Now, since $\pi$ is irreducible, the restriction $\theta|_{\rC^*(X)}=\pi$ is a pure completely positive map by \cite[Corollary 1.4.3]{arveson1969}. In particular, $\theta|_{\rC^*(X)}$ is an extreme point of $P(\rC^*(X))$. It follows readily from this that $E$ is a face in $P(A)$, so $\theta$ is in fact an extreme point of $P(A)$. By \cite[Proposition 2.2]{kleski2014boundary}, we further infer that $\theta$ must be a pure completely positive map.  Thus, $\theta$ annihilates $J$ by assumption.

(iii)$\Rightarrow$(i): Assume that $z_{\Omega_J}^Xp_J\neq 0$. Because $p_J$ is a central projection in $A^{**}$,  $z^X_{\Omega_J}p_J$ must be a non-zero projection in $A^{**}$, so there is a unital $*$-representation $\pi:A\to B(H)$ such that $\widehat\pi(z^X_{\Omega_J}p_J)=1$. In particular, if we let $\rho=\pi|_{\rC^*(X)}$, then $\widehat\rho(z^X_{\Omega_J})=1$. By Lemma \ref{L:atomic}, there is a non-zero reducing subspace $K\subset H$  for $\rho$ such that the corresponding compression $\sigma:\rC^*(X)\to B(K)$ is an irreducible $*$-representation annihilating $\rC^*(X)\cap J$.
Let $\psi:A\to B(K)$ be the unital completely positive map defined as
\[
\psi(a)=P_K \pi(a)|_K,\quad a\in A.
\]
 Then, $\psi$ agrees with $\sigma$ on $\rC^*(X)$, and $\widehat\psi(p_J)=P_K \widehat\pi(p_J)|_K=P_K\neq 0$. 
In particular, $\psi$ does not annihilate $J$. 
\end{proof}

We can now characterize when all the irreducible $*$-representations of $\B(X,J)$ are boundary representations for $\S(X,J)$.

\begin{theorem}\label{T:projcriterion}
Let $A$ be a unital $\rC^*$-algebra containing a closed two-sided ideal $J$ and an operator system $X$.  
Consider the following statements.
\begin{enumerate}[{\rm (i)}]
\item All irreducible $*$-representations of $\rC^*(X)$ annihilating $ \rC^*(X)\cap J$ are boundary representations for $X$.
\item the projections $z^X_{\at}(1-p^X_J)$ and $p_J$ are mutually orthogonal in $A^{**}$.
\item All irreducible $*$-representations of $\B(X,J)$ are boundary representations for $\S(X,J)$.
\end{enumerate}
Then, we have ${\rm (i)}+{\rm (ii)}\Leftrightarrow {\rm (iii)}$.
\end{theorem}
\begin{proof}
(i)+(ii)$\Rightarrow$(iii): 
Let $\Pi:\B(X,J)\to B(H)$ be an irreducible $*$-representation. In checking that $\Pi$ is a boundary representation for $\S(X,J)$, by virtue of Theorem \ref{T:uepnondeg} we may assume without loss of generality that $\Pi$ annihilates $\bM_2(J)$. Hence, by Lemma \ref{L:singrep} there is a unique $*$-representation $\pi:\rC^*(X)+J\to B(H)$ annihilating $J$ such that the pair $(\Pi,\pi)$ is $(X,J)$-adapted. Put $\rho=\pi|_{\rC^*(X)}$. It follows readily from \eqref{Eq:commutant} that $\pi$ and $\rho$ are  also irreducible, and hence unital. Moreover, $\rho$ annihilates $ \rC^*(X)\cap J$.
By (i), we see that $\rho$ is the only unital completely positive extension to $\rC^*(X)$ of $\pi|_X$. In turn, Proposition \ref{P:projzero} implies that $\pi|_X$ admits a unique contractive completely positive extension to $\rC^*(X)+J$. Hence,  Theorem \ref{T:singrepuep} implies that indeed $\Pi$ is a boundary representation for $\S(X,J)$ and (iii) holds.

(iii) $\Rightarrow$(i): Let $\rho:\rC^*(X)\to B(H)$ be an irreducible $*$-representation annihilating $ \rC^*(X)\cap J$. By Lemma \ref{L:charOmegaJ}, there is an irreducible $*$-representation $\pi:\rC^*(X)+J\to B(H)$ extending $\rho$ that annihilates $J$. By Lemma \ref{L:singrep}, there is a unique unital $*$-representation $\Pi:\B(X,J)\to B(H)$ annihilating $\bM_2(J)$ such that the pair $(\Pi,\pi)$ is $(X,J$)-adapted. It readily follows from \eqref{Eq:commutant} that $\Pi$ is irreducible, since $\pi$ is. Then, $\Pi$ is a boundary representation for $\S(X,J)$ by assumption. By Theorem \ref{T:singrepuep}, it follows that $\rho$ is a boundary representation for $X$, so indeed (i) holds. 

(iii) $\Rightarrow$(ii):  Assume that $z^X_{\at}(1-p^X_J)p_J\neq 0$.  By Proposition \ref{P:projzero}, there is an irreducible $*$-representation $\pi:\rC^*(X)\to B(H)$ annihilating $\rC^*(X)\cap J$ and admitting a unital completely positive extension to $\rC^*(X)+J$ that does not annihilate $J$. By Lemma \ref{L:charOmegaJ}, there is an irreducible $*$-representation $\sigma:\rC^*(X)+J\to B(H)$ annihilating $J$ and extending $\pi$. Apply Lemma \ref{L:singrep} and \eqref{Eq:commutant} to find a unique irreducible $*$-representation $\Sigma:\B(X,J)\to B(H)$ such that the pair $(\Sigma,\sigma)$ is $(X,J)$-adapted. By Theorem \ref{T:singrepuep}, we see that $\Sigma$ is not a boundary representation for $\S(X,J)$, contrary to our assumption. 
\end{proof}

We now arrive at the main result of the paper.

\begin{corollary}\label{C:counterex}
Let $A$ be a unital $\rC^*$-algebra containing a closed two-sided ideal $J$ and an operator system $X$.  Assume that 
\begin{enumerate}[{\rm (i)}]
\item $\rC^*(X)$ and $J$ have trivial intersection,
\item the projections $z^X_{\at}$ and $p_J$ are mutually orthogonal in $A^{**}$, and
\item all the irreducible $*$-representations of $\rC^*(X)$ annihilating $ \rC^*(X)\cap J$ are boundary representations for $X$.
\end{enumerate}
Then, $\S(X,J)\subset \B(X,J)$ is a counterexample to Arveson's hyperrigidity conjecture.
\end{corollary}
\begin{proof}
Note that $p_J^X=0$ in this case, by (i). On one hand, it follows from Theorem \ref{T:projcriterion} that all irreducible $*$-representations of $\B(X,J)$ are boundary representations for $\S(X,J)$. On the other hand, Corollary \ref{C:intnothyperrigid} shows that there is a unital $*$-representation of $\B(X,J)$ without the unique extension property with respect to $\S(X,J)$.
\end{proof}

In light of this result, it is natural to wonder if certain regularity assumptions on $A$ could preclude the existence of a counterexample. In other words, we seek conditions that prevent all irreducible $*$-representations from being boundary representations for $\S(X,J)$.  We identify below one such regularity condition.

\begin{theorem}\label{T:atcompat}
Let $A$ be a unital $\rC^*$-algebra containing a closed two-sided ideal $J$ and an operator system $X$. Assume that $z^X_{\at}\geq z^{A}_{\at}$ and that $J$ is not contained in $\rC^*(X)$. Then,  there is an irreducible  $*$-representation of $\B(X,J)$ that is not a boundary representation for $\S(X,J)$.
\end{theorem}
\begin{proof}
By assumption, we have that $\rC^*(X)\cap J$ is proper subset of $J$. Since norm-closed subspaces are always weakly closed, we find that $( \rC^*(X)\cap J)^{\perp\perp}\cap A=\rC^*(X)\cap J$ and $J^{\perp\perp}\cap A=J$. Therefore,  $(\rC^*(X)\cap J)^{\perp\perp}$ is a proper subset of $J^{\perp\perp}$ and $p_J-p_J^X\neq 0$.  In the weak-$*$ topology of $A^{**}$, the projections $p_J^X$ and $p_J$ are the limits of increasing, positive, contractive approximate units for $\rC^*(X)\cap J$ and $J$, respectively \cite[Proposition 2.5.8]{blecher2004operator}. Hence, both $1-p_J^X$ and $p_J$ are universally measurable elements in $A^{**}$. 
 Recall now that $p_J-p_J^X=(1-p_J^X)p_J$ by \eqref{Eq:pJX}, so that  $(1-p_J^X)p_J$ is also universally measurable by  \cite[Proposition 4.3.13]{pedersen2018}.  Applying \cite[Theorem 4.3.15]{pedersen2018}, we find  $z^A_{\at}(1-p_J^X)p_J\neq 0$.
Finally, because $z^X_{\at}\geq z^A_{\at}$, we necessarily have $z^X_{\at}(1-p_J^X)p_J\neq 0$. Hence, Theorem \ref{T:projcriterion} implies the desired conclusion.
\end{proof}

We record an immediate consequence. Recall that a unital $\rC^*$-algebra is \emph{liminal} when all its irreducible $*$-representations are finite-dimensional.

\begin{corollary}\label{C:liminal}
Let $A$ be a liminal, unital $\rC^*$-algebra. Let $X\subset A$ be an operator system and let $J\subset A$ be a closed two-sided ideal not contained in $\rC^*(X)$. Then,  there is an irreducible  $*$-representation of $\B(X,J)$ that is not a boundary representation for $\S(X,J)$.
\end{corollary}
\begin{proof}
Since all the irreducible $*$-representations of $A$ are finite-dimensional, we may invoke Lemma \ref{L:atcomp} to see that $z_{\at}^X\geq z^A_{\at}$. An application of Theorem \ref{T:atcompat} then finishes the proof.
\end{proof}

We do not know to which extent the liminality condition can be relaxed.  Because of the counterexample constructed in \cite{BDO2024}, we know that the conclusion of the previous corollary may fail to hold in the type $I$ (i.e. \emph{post}liminal) case.

\section{The compact operators}\label{S:compact}

In this final section, we examine the special case where the $\rC^*$-algebra $A$ in our construction is taken to be $B(H_0)$ for some Hilbert space $H_0$, and $J$ is taken to be the ideal $\fK(H_0)$ of compact operators thereon. 
Our first task is to show that, in this setting,  the mutual orthogonality of the projections required in Corollary \ref{C:counterex}  is implied by a von Neumann algebraic condition. 

Throughout, given a $\rC^*$-subalgebra $D\subset B(H_0)$, we let $D''\subset B(H_0)$ denote its double commutant, which coincides with the von Neumann algebra generated by $D$.
We denote by $p_\fK\in B(H_0)^{**}$ the support projection of $\fK(H_0)$. By \cite[Lemma 2.5]{CTh2022}, the projection $p_{\fK}$ has the following important property. Given a bounded linear map $\phi:B(H_0)\to B(H)$, if we define $\phi_a,\phi_s:B(H_0)\to B(H)$ by
\[
\phi_a(t)=\widehat{\phi}(tp_{\fK}) \qand \phi_s(t)=\widehat{\phi}(t(1-p_{\fK}))
\]
for each $t\in B(H_0)$, then $\phi_a$ is weak-$*$ continuous and $\phi_s$ annihilates $\fK(H_0)$. We now establish a crucial preliminary fact.

\begin{proposition}\label{P:projcompact}
Let $X\subset B(H_0)$ be an operator system.  Assume that  the von Neumann algebra $\rC^*(X)''$ contains no type $I$ factor as a direct summand. Then, the projections $z^X_{\at}(1-p_{\fK}^X)$ and $p_{\fK}$ are mutually orthogonal in $B(H_0)^{**}$.
\end{proposition}
\begin{proof}
Let $\pi:\rC^*(X)\to B(H)$ be an irreducible $*$-representation annihilating $\rC^*(X)\cap \fK(H_0)$. Let $\psi:B(H_0)\to B(H)$ be a pure unital completely positive map extending $\pi$. Since $p_{\fK}$ is central, it follows that $\psi_a$ is dominated by $\psi$ in the usual order on completely positive maps. By purity of $\psi$, there is some scalar $c\geq 0$ such that $\psi_a=c\psi$. 

If $\psi_a\neq 0$, then $c\neq 0$ and $\psi=\psi_a/c$ is weak-$*$ continuous. In turn, by the Kaplansky density theorem it readily follows that $\psi|_{\rC^*(X)''}$ is a weak-$*$ continuous irreducible $*$-representation. If $q\in \rC^*(X)''$ denote its support projection, then we have \[ q\rC^*(X)''\cong \psi(\rC^*(X)'')=\pi(\rC^*(X))''=B(H)
\] contrary to our assumption that $\rC^*(X)''$ has no type $I$ factor direct summand. We thus conclude that $\psi_a=0$ so  $\psi=\psi_s$ must annihilate $\fK(H_0)$. 

Finally, Proposition \ref{P:projzero} implies the desired conclusion.
\end{proof}

We now arrive at our desired result, which gives a sufficient condition for condition (ii) in Corollary \ref{C:counterex} to hold.

\begin{theorem}\label{T:counterexK}
Let $X\subset B(H)$ be an operator system. Assume that 
\begin{enumerate}[{\rm (i)}]
\item $\rC^*(X)$ and $\fK(H)$ have trivial intersection,
\item the von Neumann algebra $\rC^*(X)''$ contains no type $I$ factor as a direct summand, and
\item all the irreducible $*$-representations of $\rC^*(X)$ are boundary representations for $X$.
\end{enumerate}
Then, $\S(X,\fK(H))\subset \B(X,\fK(H))$ is a counterexample to Arveson's hyperrigidity conjecture.
\end{theorem}
\begin{proof}
First, Corollary \ref{C:intnothyperrigid} implies that $\B(X,\fK(H))$ has a unital $*$-representation without the unique extension property with respect to $\S(X,\fK(H))$. Consequently,  to show that $\S(X,\fK(H))\subset \B(X,\fK(H))$  serves as a counterexample to   Arveson's hyperrigidity conjecture, we must verify that all irreducible $*$-representations of $\B(X,\fK(H))$ are boundary representations for $\S(X,\fK(H))$. But this follows at once upon combining Theorem \ref{T:projcriterion} with Proposition \ref{P:projcompact}.
\end{proof}

We illustrate the previous result in some concrete cases, thereby essentially recovering the counterexample from \cite{BDO2024}.

\begin{example}\label{E:BDO}
Let $\bB_d\subset\bC^d$ denote the open unit ball, the boundary of which we denote by $\bS_d$. Let $\sigma$ denote the surface area measure on $\bS_d$. Let $$\pi:\rC(\bS_d)\to B(L^2(\bS_d,\sigma))$$ denote the usual isometric $*$-representation by multiplication operators. Let $X\subset B(L^2(\bS_d,\sigma))$ denote the image under $\pi$ of the closure of the polynomials. It is well known that $\rC^*(X)=\pi(\rC(\bS_d))$ and that this $\rC^*$-algebra contains no non-zero compact operators. In addition, since $\sigma$ is a diffuse measure, $\rC^*(X)''\cong L^\infty(\bS_d,\sigma)$ admits no type $I$ factor as a direct summand. Finally, every point in $\bS_d$ is a peak point for the polynomials (see \cite[Paragraph 10.2.3]{rudin2008}), so that every irreducible $*$-representation of $\rC^*(X)$ is a boundary representation for $X$.  Denoting the compact operators on $L^2(\bS_d,\sigma)$ simply by $\fK$, we see that $\S(X,\fK)\subset \B(X,\fK)$ is a counterexample to Arveson's conjecture,  by Theorem \ref{T:counterexK}. When $d=1$,  ignoring some minor technical differences of no importance for our purposes, we essentially obtain the counterexample found in \cite{BDO2024}.
\qed
\end{example}

We close the paper by showing that a standard ampliation trick guarantees that the assumption that $\rC^*(X)$ should have trivial intersection with $\fK(H)$ can be arranged. Given a subset $Y\subset B(H)$, we let $Y_\infty=\{\bigoplus_{n=1}^\infty y: y\in Y\}\subset B(H^{(\infty)})$. 

\begin{corollary}\label{C:ampcounterex}
Let $X\subset B(H)$ be an operator system. Assume that 
\begin{enumerate}[{\rm (i)}]
\item the von Neumann algebra $\rC^*(X)''$ contains no type $I$ factor as a direct summand, and
\item all the irreducible $*$-representations of $\rC^*(X)$ are boundary representations for $X$.
\end{enumerate}
Then, $\S(X_\infty,\fK(H^{(\infty)}))\subset \B(X_\infty,\fK(H^{(\infty)})$ is a counterexample to Arveson's hyperrigidity conjecture.
\end{corollary}
\begin{proof}
Consider the isometric, weak-$*$ continuous unital $*$-homomorphism $$\alpha:B(H) \to B(H^{(\infty)})$$ defined as
\[
\alpha(t)=\bigoplus_{n=1}^\infty t, \quad t\in B(H).
\]
Then $X_\infty=\alpha(X)$ and $\rC^*(X_\infty)=\alpha(\rC^*(X))$. By Arveson's invariance principle  \cite[Theorem 2.1.2]{arveson1969}, it follows that all irreducible $*$-representations of $\rC^*(X_\infty)$ are boundary representations for $X_\infty$. Moreover, by construction we see that $\rC^*(X_\infty)\cap \fK(H^{(\infty)})=\{0\}$. Finally, $\alpha$ is a weak-$*$  homeomorphic $*$-isomorphism onto its image,  so $\alpha(\rC^*(X))''=\alpha(\rC^*(X)'')$ does not contain a type $I$ factor as a direct summand. An application  Theorem \ref{T:counterexK} finishes the proof.
\end{proof}

\bibliographystyle{plain}
\bibliography{BDrep}
	
\end{document}